\documentclass[a4paper,11pt]{article}
\usepackage[latin1]{inputenc}
\usepackage[T1]{fontenc}
\usepackage[all]{xy}
\usepackage{amsthm}
\usepackage{graphicx,amssymb,amsmath,latexsym,amsbsy,amscd,amsfonts,amsthm,epsfig,verbatim}
\newtheorem{defi}{Definition}
\newtheorem{rem}[defi]{Remark}
\newtheorem{lemme}[defi]{Lemma}
\newtheorem{prop}[defi]{Proposition}

\usepackage[vcentering, dvips,a4paper, top=1.1in, bottom=1.1in, left=1.2in, right=1.2in ]{geometry}

\usepackage{geometry}

\title{\bfseries{A geometric description of the homology of surface bundles}}
\author{Caterina Campagnolo}
\date{}
\begin{document}
\maketitle
\begin{center}
\bfseries{Abstract}
\end{center}

In the present note we describe geometrically the homology classes in the total space of a surface bundle over a surface in terms of the holonomy map. We treat the cases where the base surface is closed or has one boundary component. 

\section{Introduction}\label{intro}
In 1986, Thurston proved \cite{thurston tores d'application} that the mapping torus of any pseudo-Anosov diffeomorphism of a hyperbolic surface is a hyperbolic $3$-manifold. Such a mapping torus has naturally the structure of a surface bundle over the circle $S^1$. From the celebrated theorem of Agol \cite{agol}, it is known that every closed oriented real hyperbolic $3$-manifold has a finite degree cover of that form.

In contrast, the question whether a surface bundle over a surface admits a real hyperbolic structure, when fibre and base are hyperbolic, is open in its full generality. What is more, even the stronger question whether there exists a word hyperbolic surface-by-surface group is still open (see the introduction of the article by Bowditch \cite{bowditch}). It is generally conjectured that it is not the case, even for the stronger question. Only the case of complex hyperbolic structures was answered by Kapovich \cite{kapovich} in the negative.

To get a better understanding of surface bundles over surfaces, we study their homology groups. While these are known abstractly, for example through the Serre spectral sequence of fibrations, we try here to give an as geometric as possible description of the homology classes, aiming at submanifolds of the total space. Indeed, the presence of an essential torus in $E$ would imply that $E$ is not negatively curved. Unfortunately, the homology does not give enough information in general to allow to prove the existence of such a torus.

In order to find this geometric description of homology, we use the Mayer-Vietoris sequence. We start with the case where the base has one boundary component, and building on it we get the case of the closed base. We restrict our attention to hyperbolic fibre.

In Section \ref{notations} we fix the notation and recall some facts on surface bundles. We compute their homology groups in Section \ref{calcul} and explain why we cannot conclude on the existence of a toroidal class.

\vspace{0.1cm}
\emph{Acknowledgements.}
This research was supported by Swiss National Science Foundation  grant  number PP00P2-128309/1. The author would like to
thank her doctoral advisor Michelle Bucher for many useful discussions at each and every stage of this work, and Mate Juhasz for helpful comments on Propositions \ref{surface a bord} and \ref{homologie fibre surface}. Many thanks to Dieter Kotschick for carefully reading a first version of this paper and suggesting an improvement of the strategy of the proof. The author is also grateful to Dan Margalit, Saul Schleimer and Richard Webb for pointing out a misformulation in the previous version of Lemma \ref{genre 2 homotope homologue} and to Thomas Koberda for critically inspecting her arguments and proposing insightful counter-examples to the previous version of this paper. Finally the author thanks Benson Farb and Nick Salter for bringing to her attention relevant references.

\section{Background}\label{notations}
An oriented surface bundle with hyperbolic fibre over any Hausdorff paracompact manifold $F\hookrightarrow E \stackrel{\pi}{\rightarrow} M$ is determined up to homeomorphism by the conjugacy class of its holonomy morphism $\pi_1(M)\rightarrow \mathrm{Mod}(F)$ (see for example Proposition 4.6 in \cite{morita}).

Here the base space $M$ will either be a closed orientable surface $\Sigma_g$ of genus g, or an orientable surface of genus g with one boundary component $\Sigma_{g, 1}$. Let us denote by 
$$\pi_1(\Sigma_g)=\left\langle a_1, ..., a_{2g}\,\vline\, \prod_{i=1}^{g}\left[a_{2i-1}, a_{2i}\right]\right\rangle,$$
where $[a_k, a_l]=a_ka_la_k^{-1}a_l^{-1}$ denotes the commutator,
$$\pi_1(\Sigma_{g, 1})=\left\langle b_1, ..., b_{2g}\right\rangle=F_{2g}$$
their fundamental groups, and let
$$f:\pi_1(\Sigma_g)\longrightarrow \mathrm{Mod}(F),\, h:\pi_1(\Sigma_{g, 1})\longrightarrow \mathrm{Mod}(F)$$
be the respective holonomy morphisms. Let us further write $f_i$ for $f(a_i)$ and $h_i$ for $h(b_i)$. The induced morphisms ${f_i}_*, {h_i}_*: H_1(F, \mathbb{Z})\rightarrow H_1(F, \mathbb{Z})$ preserve the intersection form. Hence, after fixing a basis of $H_1(F, \mathbb{Z})$, for example as in Figure \ref{fig:generateurs homologie surface}, we can see the ${f_i}_*$'s  and the ${h_i}_*$'s in the integral symplectic group $Sp(2g(F), \mathbb{Z})$. Let $\rho$ denote this canonical representation $\mathrm{Mod}(F)\rightarrow Sp(2g(F),\mathbb{Z})$.
\begin{figure}[htbp]
\centering
\includegraphics[scale=0.08]{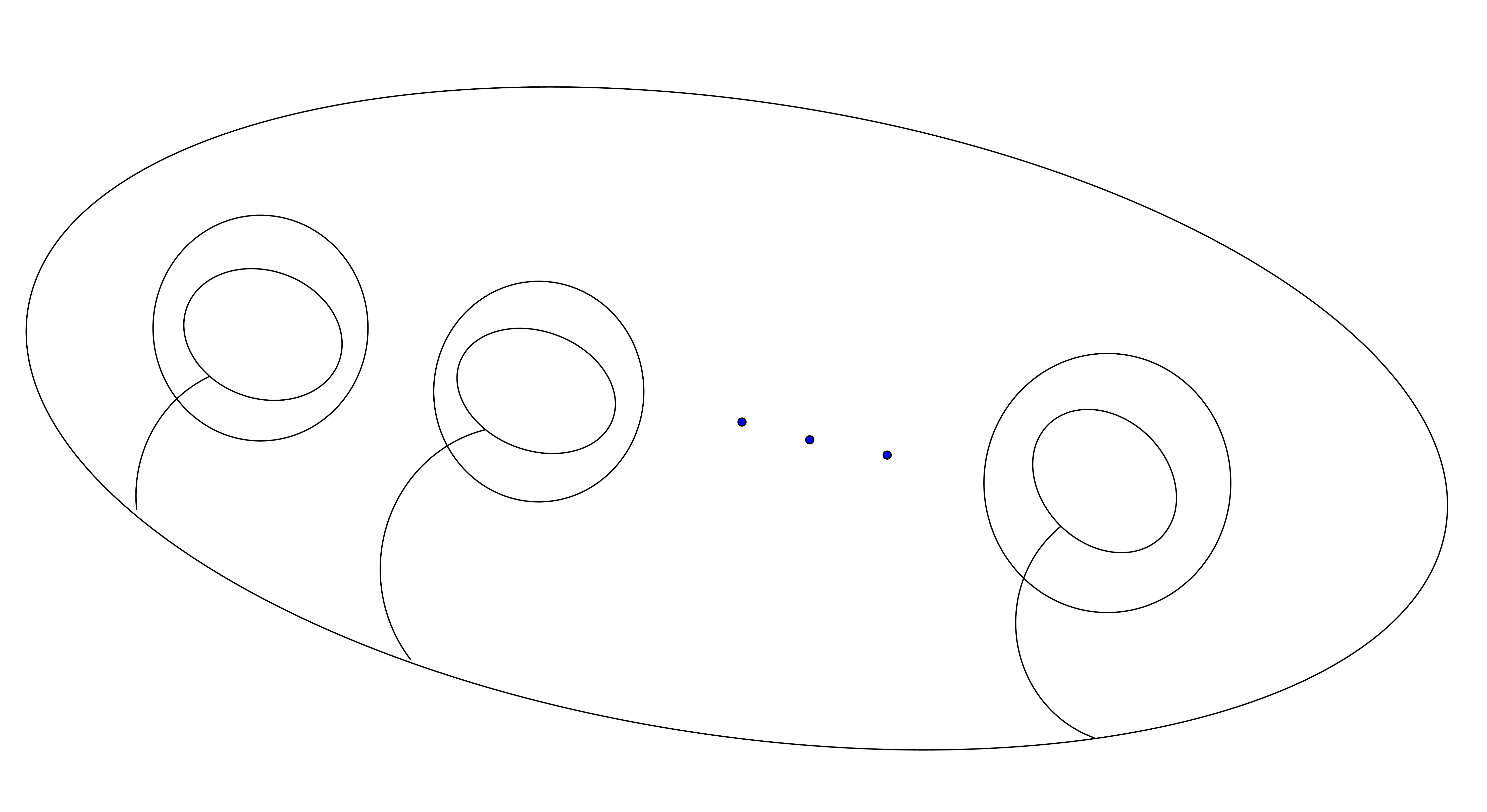}
     \caption{An integral basis for $H_1(F, \mathbb{R})$.}
     \label{fig:generateurs homologie surface}
\end{figure}

An oriented surface bundle over a surface admits a so-called tangent bundle along the fibre $T\pi$, that is the set of all tangent vectors to $E$ that are tangent to the fibres,
$$T\pi=\left\{v\in TE\,\mid \,\pi_*(v)=0\right\}.$$
The Euler class $e\in H^2(E,\mathbb{Z})$ of the bundle $E$ is defined as the Euler class of the vector bundle $T\pi$. We will denote its Poincar\'e dual by $[N]\in H_2(E, \mathbb{Z})$.

\section{Computation of the homology groups}\label{calcul}
In this section we compute the homology groups of a surface bundle over a surface in terms of the holonomy map. Though obtained by a different method, our description of the homology is an explicit geometric expression of what was computed by Morita \cite{morita spectral} for rational homology and by Cavicchioli, Hegenbarth and Repov$\mathrm{\check{s}}$ \cite{cavicchioli-hegenbarth-repovs} for integral homology, using the spectral sequence of the bundle. Note also the work of Salter \cite{salter}, who provides a detailed description of submanifolds representing homology classes in degree $2$ and $3$, focusing on bundles with holonomy lying in the Torelli subgroup. The submanifolds he describes also appear in our computations.

We start with the case where the base space is $\Sigma_{g, 1}$.

\begin{prop}\label{surface a bord}
A surface bundle $F\hookrightarrow A\rightarrow \Sigma_{g, 1}$ over the surface of genus $g$ with one boundary component is a connected $4$-manifold with boundary and has homology
$$\begin{array}{ccl}
H_0(A)&=&\mathbb{R},\\
H_1(A)&=&H_1(F)/\left\langle\beta-{h_i}_*(\beta)\,\vline\, \beta\in H_1(F), 1\leq i\leq 2g\right\rangle\oplus\bigoplus_{i=1}^{2g}H_1(S^1)\\
H_2(A)&=&H_2(F)\oplus \left\langle \left[\sum_{i=1}^{2g}I_i\times\tilde{\alpha}_i\right]\,\vline\, \alpha_i\in H_1(F), \sum_{i=1}^{2g}\alpha_i=\sum_{i=1}^{2g}{h_i}_*(\alpha_i)\right\rangle,\\
H_3(A)&=&\mathbb{R}^{2g},\\
H_4(A)&=&\{0\},
\end{array}$$
where $[\tilde{\alpha}_i]=\alpha_i\in H_1(F)$ and the index i indicates over which generator of $\pi_1(\Sigma_{g, 1})$ the cylinder $I_i\times\tilde{\alpha}_i$ lies.  When not specified, the coefficients of the homology groups are $\mathbb{R}$.
\end{prop}
\begin{rem}
The convention for the notation $I_i\times\tilde{\alpha}_i$ is that it denotes the oriented cylinder over the i-th generator of $\pi_1(\Sigma_{g, 1})$, where $I_i$ is the oriented interval, and the cylinder has base the homology class $[\tilde{\alpha}_i]$ at the first end of $I_i$ and ${h_i}_*([\tilde{\alpha}_i])$ at the second end of $I_i$.
\end{rem}
\begin{proof}[Proof]
The surface of genus $g$ with one boundary component is homotopy equivalent to the wedge sum of $2g$ circles. So the bundle $A$ is homeomorphic to an $F$-bundle over the wedge sum of $2g$ circles and hence has the same homology groups. We now consider that bundle.

We use the Mayer-Vietoris sequence. Let $P$ be the restriction of the bundle over the union of the $2g$ open intervals forming the petals of the wedge sum and $C$ be the restriction of the bundle over an open neighbourhood of the center $x_0$ of the wedge sum. The intersection $P\cap C$ has $4g$ connected components, which are $F$-bundles over small intervals. We obtain the following long exact sequence:
$$\begin{array}{cccccccc}
0&=&H_3(P)\oplus H_3(C)&\stackrel{j_3}{\longrightarrow} & H_3(A)&\stackrel{\partial _3}{\longrightarrow} &H_2(P\cap C)&\stackrel{i_2}{\longrightarrow} \\
  && H_2(P)\oplus H_2(C)&\stackrel{j_2}{\longrightarrow}  &H_2(A)&\stackrel{\partial _2}{\longrightarrow} & H_1(P\cap C)&\stackrel{i_1}{\longrightarrow}  \\
 && H_1(P)\oplus H_1(C)&\stackrel{j_1}{\longrightarrow} &H_1(A)&\stackrel{\partial _1}{\longrightarrow} & H_0(P\cap C)& \stackrel{i_0}{\longrightarrow} \\
 && H_0(P)\oplus H_0(C)&\stackrel{j_0}{\longrightarrow}  & H_0(A)&\longrightarrow &0.
\end{array}$$
The inclusions $i_P, i_C, j_P, j_C$ making the following diagram commute are given below:
$$
\xymatrix{
&\bigcup_{i=1}^{2g}\{x_i\}\times F\cup\{y_i\}\times F \simeq P\cap C   \ar[dr]^{i_C} \ar[dl]_{i_P}  &\\
P=\bigcup_{i=1}^{2g}F\times [x_i, y_i]  \ar[dr]_{j_P} &&C    \ar[dl]^{j_C}  \\
&A&
}
$$
The map $i_P$ is the identity on the $x_i$-component, and is $h_i$ on the $y_i$-component. The maps $i_C, j_P$ and $j_C$ are the canonical inclusions.

We know that
$$H_2(P\cap C)=\mathbb{R}^{4g},\, H_2(P)\oplus H_2(C)=\mathbb{R}^{2g+1},$$
$$H_1(P\cap C)=H_1(F)^{4g}, \,H_1(P)\oplus H_1(C)=H_1(F)^{2g+1},$$ and $$H_0(P\cap C)=\mathbb{R}^{4g},\, H_0(P)\oplus H_0(C)=\mathbb{R}^{2g+1}, \,H_0(A)=\mathbb{R}.$$
We compute $H_1(A)$:
$$\begin{array}{ccl}
H_1(A)&=&Im(j_1)\oplus H_1(A)/Im(j_1)=Im(j_1)\oplus H_1(A)/Ker(\partial _1)\\
	&\cong &Im(j_1)\oplus Im(\partial _1)=Im(j_1)\oplus Ker(i _0).
	\end{array}$$
The map $i_0$ is as follows:
$$\begin{array}{rrcl}
i_0: & H_0(P\cap C)& \longrightarrow & H_0(P)\oplus H_0(C)\\
	& (\tilde{x}_1, \tilde{y}_1, ..., \tilde{x}_{2g}, \tilde{y}_{2g})&\longmapsto & \left(\left(\tilde{x}_1+{h_1}_*(\tilde{y}_1), ..., \tilde{x}_{2g}+{h_{2g}}_*(\tilde{y}_{2g})\right), -\sum_{i=1}^{2g}\tilde{x}_i+\tilde{y}_i\right),
\end{array}$$
where $\tilde{x_i}, \tilde{y_i}$ denote lifts in $F$ of the points $x_i, y_i$.

The homology classes of the points $\tilde{x}_i$ and ${h_i}_*(\tilde{x}_i)$ are equal in $H_0(P)$, and similarly for $\tilde{y}_i$ and ${h_i}_*(\tilde{y}_i)$, as the fibre is connected. Hence the kernel of $i_0$ is:
$$Ker(i_0)=\left\langle (\tilde{x}_1, -\tilde{x}_1, 0, ..., 0), ..., (0, ..., 0, \tilde{x}_{2g}, -\tilde{x}_{2g})\right\rangle.$$
It is therefore isomorphic to $\mathbb{R}^{2g}$. Each of its generators corresponds in $H_1(A)$ to the circles over each circle of the wedge sum.

As $Im(j_1)\cong \left(H_1(P)\oplus H_1(C)\right)/Ker(j_1)=\left(H_1(P)\oplus H_1(C)\right)/Im(i_1)$, we investigate $i_1$:
$$\begin{array}{rrcl}
i_1: & H_1(P\cap C)& \longrightarrow & H_1(P)\oplus H_1(C)\\
	& (\alpha_1, \beta_1, ..., \alpha_{2g}, \beta_{2g})&\longmapsto & \left(\left(\alpha_1+{h_1}_*(\beta_1), ..., \alpha_{2g}+{h_{2g}}_*(\beta_{2g})\right), -\sum_{i=1}^{2g}\alpha_i+\beta_i\right)
\end{array}$$
Thus $Im(i_1)$ is equal to
$$\left\langle \left((0, ..., 0, \alpha_i , 0, ..., 0), -\alpha_i\right), \left((0, ..., 0, {h_i}_*(\beta_i) , 0, ..., 0), -\beta_i\right), \alpha_i, \beta_i\in H_1(F), 1\leq i\leq 2g\right\rangle.$$
In $(H_1(P)\oplus H_1(C))/ Im(i_1)$, we then have the following identifications:
$$((\gamma_1, ..., \gamma_{2g}), \beta)\sim ((0, ..., 0), \beta+\gamma_1+...+\gamma_{2g})\mbox{ and }$$
$$((0, ..., 0), \beta)\sim ((0, ..., {h_i}_*(\beta), 0, ..., 0), 0)\sim ((0, ..., 0), {h_i}_*(\beta)),\, 1\leq i\leq 2g.$$
This implies
$$\left(H_1(P)\oplus H_1(C)\right)/ Im(i_1)\cong H_1(F)/\left\langle\beta-{h_i}_*(\beta)\mid\beta\in H_1(F), 1\leq i\leq 2g\right\rangle,$$
and hence
$$H_1(A)\cong H_1(F)/\left\langle\beta-{h_i}_*(\beta)\mid \beta\in H_1(F), 1\leq i\leq 2g\right\rangle\oplus \bigoplus_{i=1}^{2g}H_1(S^1).$$
We compute $H_2(A)$:
$$\begin{array}{ccl}
H_2(A)&=&Im(j_2)\oplus H_2(A)/Im(j_2)=Im(j_2)\oplus H_2(A)/Ker(\partial _2)\\
	&\cong& Im(j_2)\oplus Im(\partial _2)=Im(j_2)\oplus Ker(i _1).
	\end{array}$$
The previous description of $i_1$ allows us to see that $Ker(i_1)$ is
$$\left\{(\alpha_1, \beta_1, ..., \alpha_{2g}, \beta_{2g})\mid \alpha_i=-{h_i}_*(\beta_i), 1\leq i\leq {2g},  \sum_{i=1}^{2g}\beta_i\stackrel{(*)}{=}\sum_{i=1}^{2g}{h_i}_*(\beta_i)\right\}.$$
Roughly speaking, it corresponds in $H_2(A)$ to the cylinders over the circles of the wedge sum that glue back, either to theirselves, or to a cylinder coming from another circle, possibly using part of the fibre in $C$ to match together.

As $Im(j_2)\cong \left(H_2(P)\oplus H_2(C)\right)/Ker(j_2)=\left(H_2(P)\oplus H_2(C)\right)/Im(i_2)$, we consider $i_2$:
$$\begin{array}{rrcl}
i_2: & H_2(P\cap C)& \longrightarrow & H_2(P)\oplus H_2(C)\\
	& \left(\lambda_1[F], ..., \lambda_{4g}[F]\right)&\longmapsto & \left(\left(\lambda_1[F]+\lambda_2{h_1}_*[F], ..., \lambda_{4g-1}[F]+\lambda_{4g}{h_{2g}}_*[F]\right), -\sum_{i=1}^{4g}\lambda_i[F]\right)
\end{array}$$
As ${h_i}_*([F])=[F]$ for all $1\leq i\leq 2g$, we see that 
$$Im(i_2)=\left\langle \left(([F], 0, ..., 0), -[F]\right), ..., \left((0, ..., 0, [F]), -[F]\right)\right\rangle.$$
Therefore, in $\left(H_2(P)\oplus H_2(C)\right)/Im(i_2)$, 
$$\left((0, ..., 0, \lambda [F], 0, ..., 0), \mu [F]\right)\sim \left((0, ..., 0), (\lambda +\mu)[F]\right),$$
and hence $\left(H_2(P)\oplus H_2(C)\right)/Im(i_2)\cong H_2(F)=\mathbb{R}.$
So we obtained that 
$$H_2(A)\cong H_2(F)\oplus \left\langle \left[\sum_{i=1}^{2g} I_i\times\tilde{\alpha}_i\right]\vline\alpha_i\in H_1(F), \sum_{i=1}^{2g}\alpha_i=\sum_{i=1}^{2g}{h_i}_*(\alpha_i)\right\rangle,$$
where the index $i$ indicates over which circle of the wedge sum the cylinder lies, and $[\tilde{\alpha}_i]=\alpha_i$.

Finally we compute $H_3(A)$.
This is easy, since $H_3(A)\cong Im(\partial_3)=Ker(i_2)$. By the above description of $i_2$ we compute:
$$Ker(i_2)=\left\langle\left([F], -[F], 0, ..., 0\right), ..., \left(0, ..., 0, [F], -[F]\right)\right\rangle.$$
It is therefore isomorphic to $\mathbb{R}^{2g}$. In $H_3(A)$, it corresponds to the $3$-manifolds generated by the fibre going around each circle of the wedge sum.
\end{proof}

We can use Proposition \ref{surface a bord} to treat the case where the base space is a closed surface.

\begin{prop}\label{homologie fibre surface}
Let $F\hookrightarrow E \rightarrow \Sigma_g$ be a surface bundle over the surface $\Sigma_g$ with hyperbolic fibre $F$. Then the total space $E$ is a connected $4$-dimensional manifold with homology
$$\begin{array}{ccl}
H_0(E)&=&\mathbb{R},\\
H_1(E)&=&H_1(F)/\left\langle\beta-{f_i}_*(\beta)\,\vline\, \beta\in H_1(F), 1\leq i\leq 2g\right\rangle\oplus\bigoplus_{i=1}^{2g}H_1(S^1),\\
H_2(E)&=&\left\langle [N]\right\rangle\oplus H_2(F)\oplus\\
	& &\left\langle \left[\sum_{i=1}^{2g}I_i\times\tilde{\alpha}_i\right]\,\vline\, \alpha_i\in H_1(F), \sum_{i=1}^{2g}\alpha_i=\sum_{i=1}^{2g}{f_i}_*(\alpha_i)\right\rangle/i_2\left(H_1(F)\otimes H_1(S^1)\right),\\
H_3(E)&=&\mathbb{R}^{2g}\oplus \left\langle\alpha\otimes [S^1]\, \vline\, \alpha\in H_1(F), {f_i}_*(\alpha)=\alpha, \forall 1\leq i\leq 2g\right\rangle,\\
H_4(E)&=&\mathbb{R}, 
\end{array}$$
where $[\tilde{\alpha}_i]=\alpha_i\in H_1(F)$ and i indicates over which generator of $\pi_1(\Sigma_g)$ the cylinder $I_i\times\tilde{\alpha}_i$ lies. When not specified, the coefficients of the homology groups are $\mathbb{R}$.
\end{prop}
\begin{rem}
The convention for the notation $I_i\times\tilde{\alpha}_i$ is as above:  it denotes the oriented cylinder over the i-th generator of $\pi_1(\Sigma_g)$, where $I_i$ is the oriented interval, and the cylinder has base the homology class $[\tilde{\alpha}_i]$ at the first end of $I_i$ and ${f_i}_*([\tilde{\alpha}_i])$ at the second end of $I_i$.
\end{rem}

\begin{proof}
Let $\delta$ in $\Sigma_g$ be a simple closed curve bounding a disc $D$ in $\Sigma$. Let $\Sigma_A$ denote the surface obtained from $\Sigma_g$ by cutting out $D$: its boundary is $\delta$. Let $A$ be the fibre bundle $E$ restricted to $\Sigma_A$ and $B$ the restriction of $E$ to $D$. As $D$ is contractible, the bundle $B$ is isomorphic to the trivial one. The surface with boundary $\Sigma_A$ is homotopy equivalent to the wedge sum of $2g$ circles, hence $A$ is homeomorphic to a surface bundle over the wedge sum of $2g$ circles. The restriction of $E$ to the intersection of $\Sigma_A$ and $D$ is a surface bundle over a circle. As it is the boundary of $B$, which is trivial, it is the trivial bundle as well.
By Proposition \ref{surface a bord} and the preceding remarks, we have

\begin{center}
\begin{tabular}{c|c|}
$H_*$ & $A$ \\
\hline
$0$ & $\mathbb{R}$  \\
$1$ & $H_1(F)/\left\langle\beta-{f_i}_*(\beta)\,\vline\, \beta\in H_1(F), 1\leq i\leq 2g\right\rangle\oplus\bigoplus_{i=1}^{2g}H_1(S^1)$ \\
$2$ & $H_2(F)\oplus \left\langle \left[\sum_{i=1}^{2g}I_i\times\tilde{\alpha}_i\right]\,\vline\, \alpha_i\in H_1(F), \sum_{i=1}^{2g}\alpha_i=\sum_{i=1}^{2g}{f_i}_*(\alpha_i)\right\rangle$  \\
$3$ & $\mathbb{R}^{2g}$ \\
$4$ & $\{0\}$
\end{tabular}
\begin{tabular}{c|c|c|}
$H_*$& $B$ & $A\cap B$\\
\hline
$0$&$\mathbb{R}$ & $\mathbb{R}$\\
$1$& $ H_1(F)$& $H_1(S^1)\oplus H_1(F)$\\
$2$&  $H_2(F)$&$H_2(F)\oplus H_1(S^1)\otimes H_1(F)$\\
$3$ & $\{0\}$ & $\mathbb{R}$\\
$4$ & $\{0\}$ & $\{0\}$
\end{tabular}\end{center}
We apply the Mayer-Vietoris sequence to this decomposition of $E$:
$$\begin{array}{cccccccc}
0&=&H_4(A)\oplus H_4(B)&\stackrel{j_4}{\longrightarrow} & H_4(E)&\stackrel{\partial _4}{\longrightarrow} &H_3(A\cap B)&\stackrel{i_3}{\longrightarrow} \\
  && H_3(A)\oplus H_3(B)&\stackrel{j_3}{\longrightarrow}  &H_3(E)&\stackrel{\partial _3}{\longrightarrow} & H_2(A\cap B)&\stackrel{i_2}{\longrightarrow}  \\
 && H_2(A)\oplus H_2(B)&\stackrel{j_2}{\longrightarrow} &H_2(E)&\stackrel{\partial _2}{\longrightarrow} & H_1(A\cap B)& \stackrel{i_1}{\longrightarrow} \\
 && H_1(A)\oplus H_1(B)&\stackrel{j_1}{\longrightarrow}  & H_1(E)&\stackrel{\partial_1}{\longrightarrow} & H_0(A\cap B)& \stackrel{i_0}{\longrightarrow} \\
 && H_0(A)\oplus H_0(B)&\stackrel{j_0}{\longrightarrow}  & H_0(E)&\longrightarrow &0.
\end{array}$$
As $H_0(A), H_0(B)$ and $H_0(E)$ are $1$-dimensional, we see that $i_0$ is injective, hence $\partial _1$ is the zero map, and so $j_1$ is surjective. This means that $H_1(E)=Im(j_1)\cong (H_1(A)\oplus H_1(B))/ Ker(j_1)$. But $Ker(j_1)=Im(i_1)$, so we consider $i_1$:
$$\begin{array}{cccc}
i_1:&H_1(A\cap B)&\longrightarrow & H_1(A)\oplus H_1(B)\\
	&\lambda\sigma+\phi & \longmapsto & (\overline{\phi}, -\phi)
\end{array}$$
because the image of the generator $\sigma$ of $H_1(S^1)$ is trivial in $H_1(A)$ as it goes over all circles of the wedge sum in $A$ once in each direction, and in $H_1(B)$ as it is contractible over the disc $D$. By $\overline{\phi}$ we mean the class of $\phi\in H_1(F)$ in $H_1(F)/\left\langle\beta-{f_i}_*(\beta)\,\vline\, \beta\in H_1(F), 1\leq i\leq 2g\right\rangle$.

So in $(H_1(A)\oplus H_1(B))/Im(i_1)$ we have 
$$\left(\sum_{i=1}^{2g}\lambda_i[S^1_i]+\overline{\phi}, \psi\right)\sim\left(\sum_{i=1}^{2g}\lambda_i[S^1_i]+\overline{\phi}+\overline{\psi}, 0\right),$$
hence $(H_1(A)\oplus H_1(B))/Im(i_1)\cong H_1(A)$. So $H_1(E)\cong H_1(A)$.

We turn now to $H_2(E)$:
$$\begin{array}{ccl}
H_2(E)&=&Im(j_2)\oplus H_2(E)/Im(j_2)=Im(j_2)\oplus H_2(E)/Ker(\partial_2)\\
	&\cong &Im(j_2)\oplus Im(\partial_2)=Im(j_2)\oplus Ker(i_1).
\end{array}$$
By the description of $i_1$ given above we see that
$Ker(i_1)=H_1(S^1)$, which is one-dimensional. Hence it suffices to find one non-zero preimage of it in $H_2(E)$. The Poincar\'e dual of the Euler class of $E$ is such: indeed, as the bundle $A\cap B$ is the trivial bundle, we have that the Euler class $e$ restricted to $A\cap B$ equals $\chi(F)[S^1]^*$, where $[S^1]^*$ denotes the Poincar\'e dual of $[S^1]$. The surface $F$ is hyperbolic, hence $\chi(F)\neq 0$. We compute:
$$[N]\cap[A\cap B]=\left\langle e, [A\cap B]\right\rangle=\chi(F)\left\langle [S^1]^*, [A\cap B]\right\rangle=\chi(F)[S^1]\neq 0.$$
So $\frac{1}{\chi(F)}[N]$ is a preimage in $H_2(E)$ of the generator $\sigma$ of $H_1(S^1)\subset H_1(A\cap B)$. In fact, by definition of $\partial_2$ we have 
$\partial_2([N])=\chi(F)[S^1]$.

As $Im(j_2)=(H_2(A)\oplus H_2(B))/Ker(j_2)=(H_2(A)\oplus H_2(B))/Im(i_2)$, we consider the map $i_2$.
$$\begin{array}{ccccc}
i_2: & H_2(A\cap B)& \rightarrow & H_2(A)\oplus H_2(B)&\\
	& \lambda [F]+\alpha\otimes [S^1]&\longmapsto & \left(\lambda [F]+\beta, -\lambda [F]\right),&\mbox{ where} 
\end{array}$$
$$\begin{array}{ccl}
\beta&=&\left[ \tilde{\alpha}\times I_1+{f_1}_*(\tilde{\alpha})\times I_2-{f_1^{-1}}_*{f_2}_*{f_1}_*(\tilde{\alpha})\times I_1-\left[f_2^{-1}, f_1^{-1}\right]_*(\tilde{\alpha})\times I_2+...\right.\\
& &+\prod_{i=1}^{g-1}\left[f_{2(g-i)}^{-1},f_{2(g-i)-1}^{-1}\right]_*(\tilde{\alpha})\times I_{2g-1}\\
	& &+{f_{2g-1}}_*\prod_{i=1}^{g-1}\left[f_{2(g-i)}^{-1},f_{2(g-i)-1}^{-1}\right]_*(\tilde{\alpha})\times I_{2g}\\
& &-{f_{2g-1}^{-1}}_*{f_{2g}}_*{f_{2g-1}}_*\prod_{i=1}^{g-1}\left[f_{2(g-i)}^{-1},f_{2(g-i)-1}^{-1}\right]_*(\tilde{\alpha})\times I_{2g-1}\\
& &-\prod_{i=1}^{g}\left[f_{2(g+1-i)}^{-1},f_{2(g+1-i)-1}^{-1}\right]_*(\tilde{\alpha})\times I_{2g}].
\end{array}$$

The element $\beta$ can be easily checked to belong to $H_2(A)$ by applying the corresponding holonomy maps on its summands. It is the image of $\alpha\otimes [S^1]$ with $S^1$ included in $\Sigma_A$ as the boundary curve $\delta$.

Hence in $(H_2(A)\oplus H_2(B))/Im(i_2)$ we have  
$$\left(\lambda[F]+\sum_{j}\left[\sum_i I_i\times\tilde{\alpha}_i\right]_j, \mu [F]\right)\sim \left((\lambda+\mu)[F] +\sum_{j}\left[\sum_i I_i\times\tilde{\alpha}_i\right]_j, 0\right).$$
So $\left(H_2(A)\oplus H_2(B)\right)/Im(i_2)\cong H_2(A)/i_2\left(H_1(F)\otimes H_1(S^1)\right)$.
This gives 
$$\begin{array}{ccl}
H_2(E)&=&\left\langle N\right\rangle\oplus H_2(F)\oplus\\
	& &\left\langle \left[\sum_{i=1}^{2g}I_i\times\tilde{\alpha}_i\right]\vline \alpha_i\in H_1(F), \sum_{i=1}^{2g}\alpha_i=\sum_{i=1}^{2g}{f_i}_*(\alpha_i)\right\rangle/i_2\left(H_1(F)\otimes H_1(S^1)\right).
	\end{array}$$
Finally let us compute $H_3(E)$:
$$\begin{array}{ccl}
H_3(E)&=&Im(j_3)\oplus H_3(E)/Im(j_3)=Im(j_3)\oplus H_3(E)/Ker(\partial _3)\\
	&\cong& Im(j_3)\oplus Im(\partial _3)=Im(j_3)\oplus Ker(i_2).
	\end{array}$$
As $H_4(E)\cong \mathbb{R}\cong H_3(A\cap B)$ and $\partial_3$ is injective, it is an isomorphism. Hence $i_3$ is the zero map and $j_3$ is injective. So $Im(j_3)\cong H_3(A)\oplus H_3(B)=H_3(A)=\mathbb{R}^{2g}$.
We then compute $Ker(i_2)$:
$$i_2\left(\lambda [F]+\alpha\otimes[S^1]\right)=0 \Longleftrightarrow \left(\lambda[F]+\beta, -\lambda [F]\right)=0,$$
with $\beta$ as before. Hence $\lambda=0$, and it remains $\beta=0$.
For this it is necessary that the terms lying over the same generator of $\pi_1(\Sigma_g)$ cancel each other. We sort the terms of the expression of $\beta$ by generator:
$$\begin{array}{lrcc}
I_1&\tilde{\alpha}-{f_1^{-1}}_*{f_2}_*{f_1}_*(\tilde{\alpha})&=&0\\
I_2 &{f_1}_*(\tilde{\alpha})-\left[f_2^{-1}, f_1^{-1}\right]_*(\tilde{\alpha})&=&0\\
\cdots & \cdots& &\cdots\\
I_{2g-1}& \prod_{i=1}^{g-1}\left[f_{2(g-i)}^{-1},f_{2(g-i)-1}^{-1}\right]_*(\tilde{\alpha})\\
	&-{f_{2g-1}^{-1}}_*{f_{2g}}_*{f_{2g-1}}_*\prod_{i=1}^{g-1}\left[f_{2(g-i)}^{-1},f_{2(g-i)-1}^{-1}\right]_*(\tilde{\alpha})&=&0\\
I_{2g}& {f_{2g-1}}_*\prod_{i=1}^{g-1}\left[f_{2(g-i)}^{-1},f_{2(g-i)-1}^{-1}\right]_*(\tilde{\alpha})\\
	&-\prod_{i=1}^{g}\left[f_{2(g+1-i)}^{-1},f_{2(g+1-i)-1}^{-1}\right]_*(\tilde{\alpha})&=&0
\end{array}$$
Using these equations, we obtain successively that for $\alpha\otimes [S^1]$ to belong to $Ker(i_2)$, the element $\alpha$ needs to be invariant under $f_i$ for all $1\leq i\leq 2g$. So 
$$Ker(i_2)=\left\langle\alpha\otimes [S^1]\mid \alpha\in H_1(F), {f_i}_*(\alpha)=\alpha \, \forall 1\leq i\leq 2g\right\rangle.$$
This corresponds in $H_3(E)$ to the $3$-manifolds in $E$ such that when cut along $A\cap B$, their boundary belongs to $Ker(i_2)$.
\end{proof}

The following lemma is a simple property of the surface of genus $2$, that is not true for surfaces of genus greater or equal to $3$. It opens a possibility for a better geometric adaptation of the homological information provided by the previous propositions for a bundle with genus $2$ fibre.
\begin{lemme}\label{genre 2 homotope homologue}
In the surface of genus $2$, two disjoint simple closed curves are freely homotopic if and only if they define the same homology class.
\end{lemme}

\begin{proof}
It is obvious that two homotopic simple closed curves in any surface define the same homology class. For the other direction, the disjoint simple closed curves $\alpha$ and $\beta$ are homologous if and only if $\alpha$ and $\beta$ form the boundary of a subsurface $S_1\subset \Sigma_2$. The surface $S_1$ of course has two boundary components, as has its complementary surface $S_2=\Sigma_2\setminus S_1$. The relation between the Euler characteristics of $\Sigma_2, S_1$ and $S_2$ yields the equation 
$$-2=\chi(\Sigma_2)=\chi(S_1)+\chi(S_2)=(2-2g_1-2)+(2-2g_2-2)=-2(g_1+g_2),$$
where $g_i\geq 0$ is the genus of $S_i, \, i=1, 2$. Hence $g_1+g_2=1$, and thus either $g_1$ or $g_2$ is $0$. Hence either $S_1$ or $S_2$ is the surface of genus $0$ with $2$ boundary components, which is an annulus. This provides the free homotopy between $\alpha$ and $\beta$.
\end{proof}

\begin{rem}\label{homologue pas homotope}
Our original hope was to find a non-trivial toroidal class in $H_2(E, \mathbb{R})$ among the classes of the form $\left[\sum_{i=1}^{2g}I_i\times\tilde{\alpha}_i\right]$ in order to show that $E$ is not negatively curved.

The existence of such a class requires the existence of a sequence of classes $[\alpha_{i_1}], ...,  [\alpha_{i_k}]$ in $H_1(F)$ such that ${f_{i_j}}_*([\alpha_{¡_j}])=[\alpha_{i_{j+1}}]$ for $1\leq j\leq k-1$ and ${f_{i_k}}_*([\alpha_{i_k}])=[\alpha_{i_1}]$. Indeed, such a sequence would produce a long cylinder in $E$ whose two ends would be two representatives of the class $[\alpha_{i_1}]$. If these two representatives are disjoint and the fibre is of genus $2$, then Lemma \ref{genre 2 homotope homologue} ensures that the cylinder can be closed back into a torus. But if they intersect, one needs to add some handles and the genus of the submanifold increases.

Note that the existence of $\left\{[\alpha_{i_1}], ...,  [\alpha_{i_k}]\right\}\subset H_1(F)$ with ${f_{i_j}}_*([\alpha_{¡_j}])=[\alpha_{i_{j+1}}]$ for $1\leq j\leq k-1$ and ${f_{i_k}}_*([\alpha_{i_k}])=[\alpha_{i_1}]$ as above is equivalent to the existence of an element $\phi={f_{i_k}}_*\cdots {f_{i_1}}_*$ with eigenvalue $1$ in $\rho(Im(f))$.

However, the existence of a fixed homology class, even in a surface of genus $2$, does not at all imply the existence of a fixed curve. Just take for $\phi$ a pseudo-Anosov element of $\mathrm{Mod}(F)$ lying in the Torelli subgroup.

\end{rem}

\bibliographystyle{alpha}

{\scshape Section de Math\'ematiques, Universit\'e de Gen\`eve,  2-4 rue du Li\`evre,  Case postale 64, 1211 Gen\`eve 4, Suisse,}
\emph{E-mail address: }{\ttfamily caterina.campagnolo@unige.ch}

\end{document}